\documentclass[12pt,reqno]{amsart}
\usepackage{amsmath,amsthm,amscd,amsfonts,amssymb,color}
\usepackage{cite}
\usepackage[mathscr]{eucal}
\usepackage{mathtools}

\usepackage[bookmarksnumbered,colorlinks,plainpages]{hyperref}
\newtheorem{theorem}{Theorem}[section]
\newtheorem{lemma}[theorem]{Lemma}

\theoremstyle{definition}
\newtheorem{definition}[theorem]{Definition}

\newtheorem{Open Prob}[theorem]{Open Problem}
\theoremstyle{remark}

\numberwithin{equation}{section}
\def\DJ{\leavevmode\setbox0=\hbox{D}\kern0pt\rlap{\kern.04em\raise.188\ht0\hbox{-}}D}
\begin{document}

\title[Proximal contraction]{A note on best proximity point for proximal contraction}

\author[S.\ Som]
{Sumit Som}

\address{           Sumit Som,
                    Department of Mathematics,
                    School of Basic and Applied Sciences, Adamas University, Barasat-700126, India.}
                    \email{somkakdwip@gmail.com}

\subjclass {$47H10$}

\keywords{Best proximity point, proximal contraction, Banach contraction principle}

\begin{abstract}
In the year 2011, S.Basha \cite{BS} introduced the notion of proximal contraction in a metric space $X$ and study the existence and uniqueness of best proximity point for this class of mappings. Also, the author gave an algorithm to achieve this best proximity point. In this paper, we show that the best proximity point theorem can be proved by Banach contraction principle.
\end{abstract}

\maketitle

\section{\bf{Introduction}}
Metric fixed point theory is an essential part of Mathematics as it gives sufficient conditions which will ensure the existence of solutions of the equation $F(x)=x$ where $F$ is a self mapping defined on a metric space $(M,d).$ Such a solution is called a fixed point of the mapping $F.$ Banach contraction principle for standard metric spaces is one of the important results as it gives a class of mappings to have a unique fixed point. It also has a lot of applications in the area of differential equations, integral equations. Let $A,B$ be non-empty subsets of a metric space $(M,d)$ and $Q:A\rightarrow B$ be a non-self mapping. A necessary condition, to guarantee the existence of solutions of the equation $Qx=x,$ is $Q(A)\cap A\neq \phi.$ If $Q(A)\cap A= \phi$ then the mapping $Q$ has no fixed points. In this case, one seek for an element in the domain space whose distance from its image is minimum i.e, one interesting problem is to $\mbox{minimize}~d(x,Qx)$ such that $x\in A.$ Since $d(x,Qx)\geq d(A,B)=\inf~\{d(x,y):x\in A, y\in B\},$ so, one search for an element $x\in A$ such that $d(x,Qx)= d(A,B).$ Best proximity point problems deal with this situation. Authors usually discover best proximity point theorems to generalize fixed point theorems in metric space. In the year 2011, Basha \cite{BS} investigated about the existence of best proximity points about proximal contractions. Also, in that article, author gave an algorithm to achieve the unique best proximity point. In this article, we show that the conclusion of the best proximity point theorem for proximal contraction can be achieved from the fixed point theory by applying the Banach contraction principle. Also, this implies the picard sequence will also work to find the best proximity point.

\section{\bf{Main results}}

We first recall the following definition of proximal contraction from \cite{BS}.

\begin{definition}\cite{BS}
Let $(A, B)$ be a pair of nonempty subsets of a metric space $(M,d).$ A mapping $f:A\rightarrow B$ is said to be a proximal contraction if there exists a non-negetive number $\alpha<1$ such that   \[
\begin{rcases}
 d(u_1, f(x_1))= d(A, B)\\
 d(u_2, f(x_2))= d(A, B)
 \end{rcases}
  {\Longrightarrow d(u_1,u_2)\leq \alpha d(x_1,x_2)}
  \] for all $u_1, u_2, x_1, x_2 \in A,$ where $d(A,B) =  \inf\Big\{d(x, y): x\in A,\ y\in B\Big\}.$
\end{definition}
The following notations will be needed.
Let $(M,d)$ be a metric space and $A,B$ be nonempty subsets of $M.$ Then
$$A_0=\{x\in A: d(x,y)=d(A,B)~\mbox{for some}~y\in B\}.$$
$$B_0=\{y\in B: d(x,y)=d(A,B)~\mbox{for some}~x\in A\}.$$

\begin{definition}\cite{BS}
Let $(M,d)$ be a metric space and $A,B$ be two non-empty subsets of $M.$ Then $B$ is said to be approximatively compact with respect to $A$ if for every sequence $\{y_n\}$ of $B$ satisfying $d(x,y_n)\rightarrow d(x,B)$ as $n\rightarrow \infty$ for some $x\in A$ has a convergent subsequence.
\end{definition}

We need the following result from \cite{FE}.
\begin{lemma}\cite{FE}\label{b}
Let $(A,B)$ be a nonempty and closed pair of subsets of a metric space $(X,d)$ such that $B$ is approximatively compact with respect to $A.$ Then $A_0$ is closed.
\end{lemma}

In \cite{BS}, S.Basha proved the following best proximity point result.
\begin{theorem}\cite{BS}\label{a}
Let $A,B$ be non-void, closed subsets of a complete metric space $(M,d)$ such that $B$ is approximatively compact with respect to $A.$ Moreover, assume that $A_0,B_0$ are non-void. Let the non-self mapping $T:A\rightarrow B$ be a proximal contarction such that $T(A_0)\subseteq B_0.$ Then there exists an unique best proximity point for $T$ and for any $x_0\in A_0$ the sequence defined by $d(x_{n+1},T(x_n))=d(A,B)$ converges to the best proximity point.
\end{theorem}

Now we like to state our main result.
\begin{theorem}
Theorem \ref{a} follows from the Banach contraction principle.
\end{theorem}

\begin{proof}
Let $x\in A_0.$ As $T(A_0)\subseteq B_0,$ so, $T(x)\in B_0.$ That means there exists $y\in A_0$ such that $d(y,T(x))=d(A,B).$ Now we will show that $y\in A_0$ is unique. Suppose there exists $y_1,y_2\in A_0$ such that $d(y_1,T(x))=d(A,B)$ and $d(y_2,T(x))=d(A,B).$ Now as $T:A\rightarrow B$ is a proximal contraction so we have,
$$d(y_1,y_2)\leq \alpha d(x,x)$$
$$\Rightarrow y_1=y_2.$$
Let $S:A_0\rightarrow A_0$ be defined by $S(x)=y.$ Now we will show that $S$ is a contraction mapping. Let $x_1,x_2\in A_0.$ As $d(S(x_1),T(x_1))=d(A,B)$ and $d(S(x_2),T(x_2))=d(A,B)$ so we have,
$$d(S(x_1),S(x_2))\leq \alpha d(x_1,x_2).$$ This shows that $S:A_0\rightarrow A_0$ is a Banach contraction mapping. Now, from lemma \ref{b}, we can say $A_0$ is closed. So, $A_0$ is a complete metric space. Then by Banach contraction principle the mapping $S$ has an unique fixed point $z\in A_0.$ Now $d(z,T(z))=d(S(z),T(z))=d(A,B).$ This shows that $z$ is a best proximity point for $T.$ Uniqueness follows from the definition of proximal contraction. Also, we can conclude that for any $x_0\in A_0$ the sequence $\{S^{n}(x_0)\}$ will converge to the unique best proximity point of $T.$
\end{proof}

\end{document}